\documentclass[12pt,reqno]{amsart}
\usepackage{url}
\usepackage{amsthm}
\usepackage[top=1.2in, bottom=1.2in, left=1in, right=1in]{geometry}
\usepackage{amsmath}
\usepackage{amssymb}
\usepackage{amsfonts}

\usepackage{enumitem}
\usepackage{mathscinet}
\usepackage{lipsum}
\usepackage[english]{babel}
\usepackage[autostyle]{csquotes}
\usepackage{bbold}
\usepackage{hyperref}

\usepackage{subfigure}

\hypersetup{
    colorlinks=true,
    linkcolor=red,
    citecolor=blue,
    }

\usepackage{graphicx}

\usepackage[utf8]{inputenc}
\usepackage[english]{babel}

\usepackage{color}
\usepackage{comment}

\newtheorem{thm}{Theorem}[section]
\newtheorem{definition}[thm]{Definition}
\newtheorem{theorem}[thm]{Theorem}
\newtheorem*{oseledec*}{Oseledec Theorem}
\newtheorem{cor}[thm]{Corollary}

\newtheorem{lemma}[thm]{Lemma}
\newtheorem{prop}[thm]{Proposition}

\makeatletter
\def\moverlay{\mathpalette\mov@rlay}
\def\mov@rlay#1#2{\leavevmode\vtop{%
   \baselineskip\z@skip \lineskiplimit-\maxdimen
   \ialign{\hfil$\m@th#1##$\hfil\cr#2\crcr}}}
\newcommand{\charfusion}[3][\mathord]{
    #1{\ifx#1\mathop\vphantom{#2}\fi
        \mathpalette\mov@rlay{#2\cr#3}
      }
    \ifx#1\mathop\expandafter\displaylimits\fi}
\makeatother

\newcommand{\nocontentsline}[3]{}
\newcommand{\tocless}[2]{\bgroup\let\addcontentsline=\nocontentsline#1{#2}\egroup}

\usepackage{wasysym}

\def\Vol{\ensuremath{\mathrm{Vol}}}

\title[Tubular dimension]{Tubular dimension
: Leaf-Wise Asymptotic Local Product Structure, and Entropy and Volume Growth
}

\author{Snir Ben Ovadia}

\newcommand{\Addresses}{{
  \bigskip
  \footnotesize

  S.~Ben Ovadia, \textsc{Department of Mathematics, Pennsylvania State University, State College, Pennsylvania 16801, United States}. \textit{E-mail address}: \texttt{snir.benovadia@psu.edu}
}}

\begin{document}
\maketitle
\begin{abstract}
We introduce the notion of tubular dimension, and give a formula for it. As an application we show that every invariant measure of a $C^{1+\gamma}$ diffeomorphism of a closed Riemannian manifold admits an asymptotic local product structure for conditional measures on intermediate foliations of unstable leaves. As a second application, we prove a bound on the gap between any two consecutive conditional entropies, in the form of volume growth. As a third application, for certain $C^\infty$ maps we compute all conditional entropies for the measure of maximal entropy; And in particular as a consequence, in a follow-up paper we compute the Hausdorff dimension of the equilibrium measure of holomorphic endomorphisms of $\mathbb{C}\mathbb{P}^k$, $k\geq 1$, giving a solution to the Binder-DeMarco conjecture, and answering a question of Forn{\ae}ss and Sibony. 
\end{abstract}

\tableofcontents

\section{Introduction and main results}
\subsection{Motivation}\label{moti}
Given a dynamical system $(M,f)$, invariant probability measures are an important object, which represents the system in its different ``equilibria", where the probability of an event does not change with time. 

In particular, smooth dynamical system (i.e $M$ is a closed Riemannian manifold, and $f\in\mathrm{Diff}^{1+}(M)$) are important as models of physical systems, where we expect the time evolution map to be smooth. Such systems are endowed with a geometric structure, and a natural question is to understand the interplay between the geometry of the space, and ``the geometry of invariant measures". For example, are invariant measures exact dimensional? Can we find singular sub-manifolds on which an invariant measure disintegrates in a non-atomic manner? What is the relationship between Lyapunov exponents, which are a dynamical quantity which measures chaos, and entropy, a measure theoretic measurement of chaos, and the dimension of invariant measures? This question relates to the understanding of the geometry of Bowen balls and balls.

Thanks to Pesin theory \cite{Pesin77} we know that every measure which admits a positive Lyapunov exponent a.e admits a singular smooth sub-manifold a.e, called an unstable leaf. Studying the measure theoretic properties of the conditional measures on unstable leaves has physical importance (see for example the study of SRB and physical measures \cite{YoungSRBsurvey}).

In \cite{LedrappierYoungII}, Ledrappier and Young proved that for any measure with positive entropy, the conditional measures on unstable leaves have a rich structure, which includes exact point-wise dimension. Moreover, conditional measures on unstable leaves can be disintegrated further into ``stronger" unstable leaves, corresponding to larger Lyapunov exponents. Ledrappier and Young gave a complete formula to compute the point-wise dimension of each conditional measure on a strong unstable foliation, where the formula is in terms of Lyapunov exponents and conditional entropy. The notion of conditional entropy is a natural extension of the metric entropy: In \cite{BK} Brin and Katok showed that the entropy of a measure can be computed by the asymptotic exponential decay rate of the measure of Bowen balls; Conditional entropy extends this notion by considering asymptotic exponential decay rate of the conditional measure of Bowen balls. Thus Ledrappier and Young put together in one formula three of the most important quantities of smooth dynamical systems: dimension, entropy, and Lyapunov exponents.

A remarkable aspect of this result is the fact that in particular the conditional measures on unstable foliations are non-atomic in the presence of positive entropy (and similarly for strong unstable foliation and conditional entropy). This implies that the invariant measure is ``localized" on smooth geometric objects.

To illustrate this, consider the strong unstable foliation inside an unstable leaf, on which the conditionals admit a positive dimension- i.e ``many" typical points inside such a leaf, but not full dimension. Now consider any other typical foliation, for example a foliation by planes in the chart of the unstable leaf. Then each plane intersects each strong leaf at finitely many points (typically), and the less probable case is that this intersection happens at a typical point.

Another heuristic of the non-triviality of that fact, is the lack of transverse foliations which correspond to weakly expanding directions in an unstable leaf. Their lacking implies the difficulty to find any foliations transverse to the strong unstable leaves on which the conditionals are non-atomic. 

The proof of Ledrappier and Young involves computing the measure of transverse balls (see Definition \ref{Defs1}), which are balls saturated by leaves of the strong unstable foliation. One can think of transverse balls as ``long" in the direction of the strong expansion, and ``narrow" in the direction of the weak expansion. 
\begin{align*}
	&\text{Q1: What dynamically significant object can we define which is ``long" in the weak direction?}\\
&\text{Q2: How can we compute the measure of such an object?}\\
&\text{Q3: Can we find such an object with a nice geometric description as well?}
\end{align*}
We call such an object a {\em tube} (see Definition \ref{tubes}), referring to its geometric proportions, and to its desired nice geometric description. Questions 1-3 are particularly interesting in light of the lacking of dynamical foliations transverse to the strong unstable leaves. 

The idea behind the notion of tubes, is to find an object such that if the weakly expanding direction were to be integrable, the tube would be almost a pluck of that invariant foliation. However, even in that case, there is no guarantee that the conditional measures on that foliation would be non-atomic. Moreover, even in the case that there is such a foliation, and that the conditional measures on this foliation have non-trivial structure, there is still no guarantee that they would admit an {\em asymptotic local product structure}, as illustrated in Figure 3 below.

\medskip
In \cite{AsymProd}, Pesin, Barreire, and Schmeling showed that hyperbolic measures (i.e no zero Lyapunov exponents, and at least one positive exponent and one negative exponent) are exact dimensional (rather than only their conditionals). They showed further that hyperbolic measures admit asymptotic local product structure, which we explain below.  
\begin{figure}[h]
    \centering
    \subfigure{Fig.1}{\includegraphics[width=0.24\textwidth]{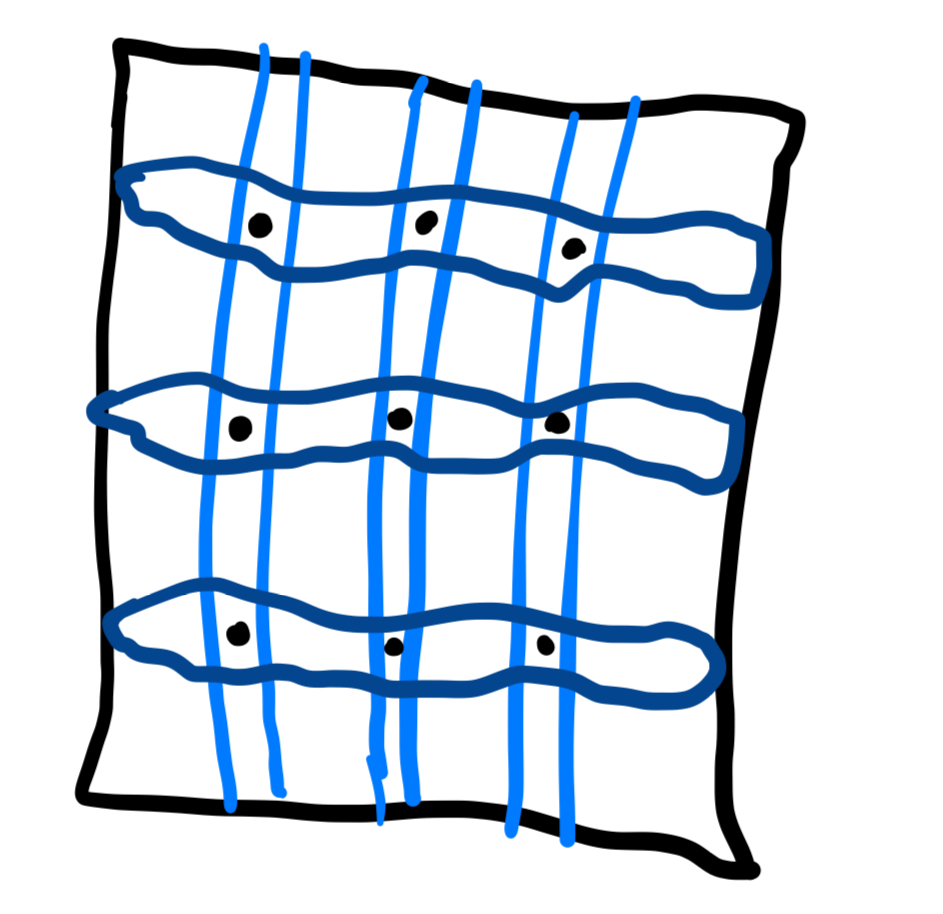}} 
    \subfigure{Fig. 2}{\includegraphics[width=0.24\textwidth]{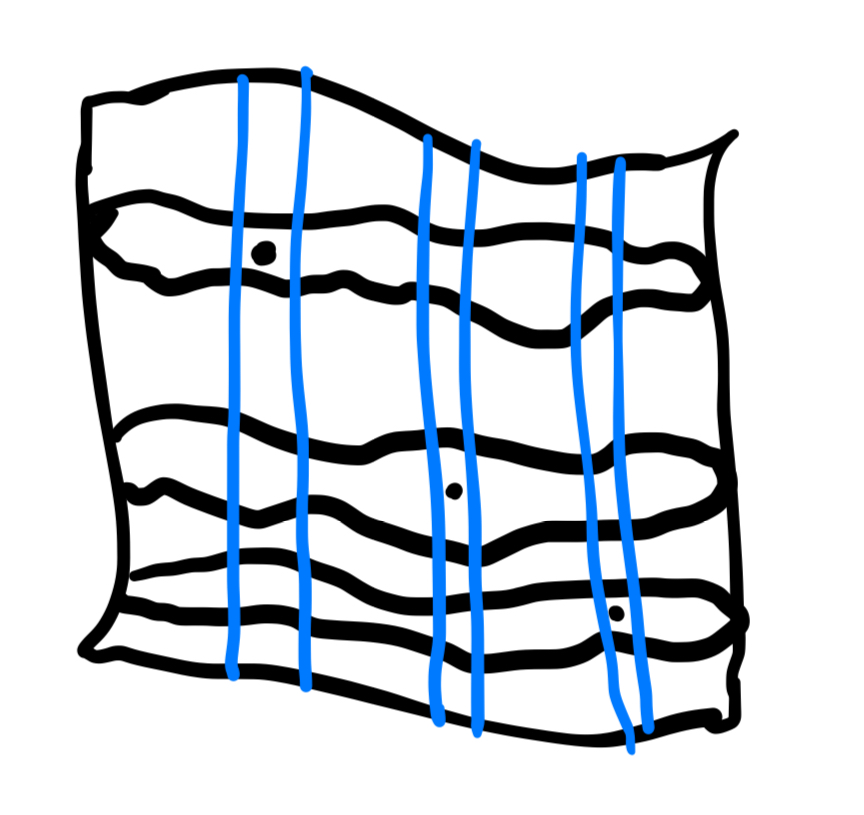}} 
    \subfigure{Fig. 3}{\includegraphics[width=0.24\textwidth]{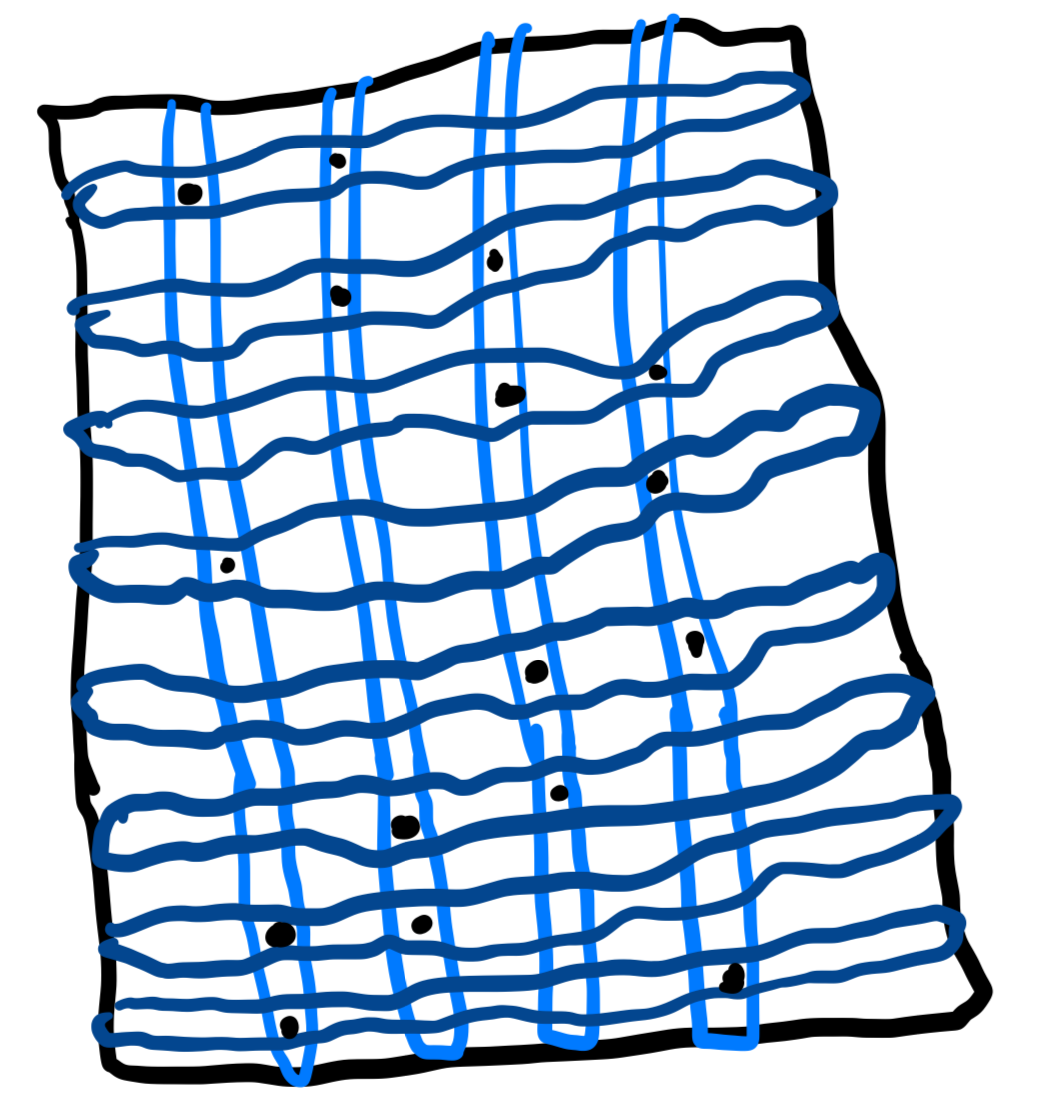}}
\end{figure}

Assume that we have a discrete measure in a chart where each atom has equal mass. We can then cover the full measure set of atoms by 3 different covers: vertical cylinders (called verticals), horizontal cylinders (called horizontals), and intersections of both (called boxes). In the three figures above we illustrate the distribution of a measure with, and without a product structure.

In figure 1 we see nine total atoms, covered by at least nine boxes. Each vertical contains exactly three atoms, and at least three 
verticals are required in order to cover the measure. Similarly for the horizontals. The notion of product structure in this context means that the number of boxes needed to cover the measure is the product of the number verticals needed and horizontals needed. The intersection of every vertical which contains an atom with a horizontal which contains an atom, is a box which contains an atom.

In figure 2, we see the extreme opposite of a product structure: a diagonal measure. Each cylinder contains exactly one atom, and the number of horizontals needed to cover the measure, times the number of verticals needed to cover the measure, is three times bigger than the number of boxes needed. For every horizontal which contains an atom, only one vertical intersects it in an atom. However, diagonal measures are not the only case with no product structure. In figure 3 we see another example with no product structure, but which is not diagonal. 

The notion of asymptotic local product structure is that for almost every point, at small enough scale, the number of boxes needed to cover the measure (or a large measure subset of ``good" points) is exponentially large (in the presence of positive entropy), while the number of boxes is up to a sub-exponential factor the product of the number of verticals and the number of horizontals.

In \cite{AsymProd} the authors prove that every hyperbolic measure has an asymptotic local product structure, where the horizontals are Bowen balls to the past, and the verticals are Bowen balls to the future. This statement gives useful information on the geometry of the localization of hyperbolic measures. Their proof relies on the fact that both families of horizontals and verticals consist of dynamically meaningful objects, and on the fact that Bowen balls to the past are saturated by unstable leaves (resp. stable leaves for Bowen balls to the future).
\begin{align*}
	&\text{Q4: Can we also show an asymptotic local product structure by tubes and transverse balls?}
\end{align*}

\subsection{Main results}
In this paper we answer questions Q1-Q4 from \textsection \ref{moti}, and offer a couple of applications.

\begin{enumerate}
	\item We introduce {\em tubes} (see Definition \ref{tubes}), which are elongated in the weakly expanding direction, with exponential eccentricity. In \textsection \ref{totallyTubular} we show geometric measure theoretic properties of tubes (e.g they form a differentiation basis), and we compute their measure as a function of their eccentricity, which is called the {\em tubular dimension} (see Definition \ref{tEnt}).

\item As a first application to the formula of the tubular dimension, we prove that for every $f$-invariant measure, the conditional measures on unstable leaves admit asymptotic local product structure (see \textsection \ref{AympLocl5}). 

\item As an additional application, in \textsection \ref{entGap} we give a bound to the {\em entropy gap} $h_{i+1}(x)-h_i(x)$ for any two consecutive conditional entropies (see Definition \ref{Defs0}) for any $f$-invariant measure, where the gap is bounded by the {\em volume growth} (see Definition \ref{UniVolGrow}). 

\item In Corollary \ref{fin}, using the bounds on the entropy gap, for $C^\infty$ diffeomorphisms which satisfy a certain super-additive relationship between the volume growth of disks in different dimensions, we compute the conditional entropies of the measure of maximal entropy (see Definition \ref{Defs0.5}).
\end{enumerate}

In particular, in a follow up paper we show that for holomorphic endomorphisms of $\mathbb{C}\mathbb{P}^k$, $k\geq 1$, the assumptions of Corollary \ref{fin} hold, and as a consequence we provide a formula for the Hausdorff dimension of the measure of maximal entropy. This gives an answer to a question of Forn\ae ss and Sibony in their list of fundamental open problems in higher dimensional complex analysis and complex dynamics (\cite[Question~2.17]{FS01}), and proves the Binder-DeMarco conjecture (\cite[Conjecture~1.3]{BinderDeMarcoConj}).

\section{Setup and definitions}

Let $M$ be a closed Riemannian manifold of dimension $d\geq 2$, and let $f\in \mathrm{Diff}^{1+\gamma}(M)$ with $\gamma>0$.

\begin{definition}[Pesin blocks]\label{Defs0.5}
Let $\mu$ be an ergodic $f$-invariant probability measure which admits $u\geq 1$ distinct positive Lyapunov exponents. Set $\underline\chi:=((\chi_1,k_1),\ldots ,(\chi_{u},k_{u}))$, where $k_i$ is the dimension of the Oseledec subspace corresponding to $\chi_i$ for all $1\leq i\leq u$, and $\chi_i>\chi_{i+1}$ for all $1\leq i\leq u-1$.
\begin{enumerate}
	\item Let $0<\tau\leq \tau_{\underline{\chi}}:= \frac{1}{100d}\min\{\chi_u,\chi_{i}-\chi_{i+1}: i\leq u-1\}$, and let $C_{\underline{\chi},\tau}(\cdot)$ be the Lyapunov change of coordinates for points in $\mathrm{LR}_{\underline{\chi}}=\{\text{Lyapunov regular points with an index }\underline\chi\}$ (see \cite{KM}).
	\item Let $\mathrm{PR}_{\underline{\chi}}=\{x\in\mathrm{LR}_{\underline{\chi}}:\limsup_{n\to\pm\infty}\frac{1}{n}\log\|C_{\underline{\chi},\tau}^{-1}(f^n(x))\|=0,\forall 0<\tau\leq \tau_ {\underline{\chi}}\}$, the set of {\em $\underline{\chi}$-Pesin regular} points which carries $\mu$. $\mathrm{PR}:=\bigcup_{\underline\chi}\mathrm{PR}_{\underline\chi}$ is called the set of {\em Pesin regular} points.
\item Given $x\in \mathrm{PR}_{\underline{\chi}}$, let $E_j(x)$ be the Oseledec subspace of $x$ corresponding to $\chi_j$.
\item A {\em Pesin block} $\Lambda^{(\underline{\chi},\tau)}_\ell$ is a subset of $\bigcup_{|\underline{\chi}'-\underline\chi|_\infty< \tau}\mathrm{PR}_{\underline\chi'}$ which is a level set $[q_\tau\geq\frac{1}{\ell}]$ of a measurable function $q_\tau: \bigcup_{|\underline{\chi}'-\underline\chi|_\infty< \tau}\mathrm{PR}_{\underline\chi'}\to (0,1)$ s.t (a) $\frac{q_\tau\circ f}{q_\tau}=e^{\pm \tau}$, (b)  $q_\tau(\cdot)\leq\frac{1}{\|C^{-1}_{\underline\chi,\tau}(\cdot)\|^\frac{d}{\gamma}}$
. Often we omit the subscript $\ell$ when the dependence on $\ell$ is clear from the context.
\item Given $i\leq u$, there exists a {\em strong unstable leaf} $W^i(x)$ tangent to $\oplus_{j\leq i}E_j(x)$, for $\mu$-a.e $x$ (see \cite{RuelleFoliations} for example). The laminations $W^1>\ldots>W^u$ are called {\em the intermediate foliations}.
\end{enumerate}	
\end{definition}

\begin{definition}[Conditional entropy and dimension \cite{LedrappierYoungII}]\label{Defs0} Let $\mu$ be an ergodic $f$-invariant probability measure which admits $u\geq 1$ distinct positive Lyapunov exponents. Let $\xi_i$, $i=1,\ldots ,u$ be the increasing measurable partitions subordinated to the intermediate foliations constructed in \cite[\textsection~9]{LedrappierYoungII}, and let $\{\mu_{\xi_i}(x)\}$ be the system of conditional measures given by the Rokhlin disintegration theorem. 
\begin{enumerate}
		\item Given $\xi_i(x)$, for $y\in \xi_i(x)$ write $d_n^i(y,x):=\max_{0\leq j\leq n} d_{f^j[\xi_i(x)]}(f^j(y),f^j(x))$. 
		\item $B_{\xi_i(x)}(x,n,\epsilon):=\{y\in\xi_i(x):d^i_n(y,x)\leq \epsilon\}$.
		\item $h_i(x):=\lim_{\epsilon\to0}\limsup\frac{-1}{n}\log\mu_{\xi_i(x)}(B_{\xi_i(x)}(x,n,\epsilon))$.
		\item $d_i(x):=\limsup_{r\to0}\frac{\mu_{\xi_i(x)}B_{\xi_i(x)}(x,r)}{\log r}$.
	\end{enumerate}
\end{definition}

\begin{theorem}[\cite{LedrappierYoungI,LedrappierYoungII}] Let $\mu$ be an ergodic $f$-invariant probability measure which admits $u\geq 1$ distinct positive Lyapunov exponents. Then for all $1\leq i\leq u$,
\begin{enumerate}
	\item $h_i(x)= \lim_{\epsilon\to0}\liminf\frac{-1}{n}\log\mu_{\xi_i(x)}(B_{\xi_i(x)}(x,n,\epsilon)) $ is constant a.e,
	\item $d_i(x)=\liminf_{r\to0}\frac{\mu_{\xi_i(x)}B_{\xi_i(x)}(x,r)}{\log r}$ is constant a.e,
	\item for $i\leq u-1$, $d_{i+1}-d_i=\frac{h_{i+1}-h_i}{\chi_{i+1}}$,
	\item $h_u=h_{\mu}(f)$.
\end{enumerate}
	
\end{theorem}

\begin{definition}[Transverse structure]\label{Defs1}
Let $\mu$ be an ergodic $f$-invariant Borel probability measure, and assume that $\mu$ admits $u\geq 2$ distinct positive Lyapunov exponents. 
\begin{enumerate}
\item Given $1\leq i\leq u-1$ and a partition element $\xi_{i+1}(x)$, Ledrappier and Young construct in \cite[\textsection~11]{LedrappierYoungII} a {\em transverse metric} $d^{T}_{i+1}(\cdot,\cdot)$ on $\xi_{i+1}(x)$.\footnote{The transverse metric is defined for $\eta_{i+1}/\eta_i$ in \cite[\textsection~11]{LedrappierYoungII}, and extends to $\xi_{i+1}$ through the extension of the Lipschitz continuous projection map on a compact set.} Their construction uses the fact the strong foliation $\xi_i>\xi_{i+1}$ is Lipschitz, which they prove for $f\in \mathrm{Diff}^2(M)$. In \cite{AsymProd}, Barreira, Pesin, and Scmeling extend the proof to our setup, where $f\in \mathrm{Diff}^{1+\gamma}(M)$, and the strong foliation $\xi_i$ is restricted leaves of points in a Pesin block $\Lambda^{(\underline\chi,\tau)}_\ell$, where $\tau>0$ is sufficiently small w.r.t $\chi_u$.
\item A {\em transverse ball} is a ball in the transverse metric $d_{i+1}^T(\cdot,\cdot)$.
\item By \cite[\textsection~8.3]{LedrappierYoungII}, for each $\xi_{i+1}$ and $x\in\Lambda^{(\underline\chi,\tau)}_\ell$ there exists a Lipschitz change of coordinates $\mathcal{O}_x$ s.t all $W^i$-unstable leaves of $\Lambda^{(\underline\chi,\tau)}_\ell$ become planes. $\mathcal{O}_x$ depends only on the leaf $\xi_{i+1}(x)$. Let $\pi_{i+1}$ be the projection onto the $\mathbb{R}^{\mathrm{dim}W^{i+1}-\mathrm{dim}W^i}$-coordinates and $\pi_{i}$ be the projection onto the $\mathbb{R}^{\mathrm{dim}W^i}$-coordinates. 
\end{enumerate}
 \end{definition}
 
 Assume w.l.o.g once and for all throughout this paper that $\mu$ is ergodic and $\chi_u>1$.
\begin{definition}[Pucks, intermediate entropy, and scaling parameters]\label{Defs2}
Given $1\leq i\leq u-1$, fix $\Delta>0$, the {\em window scale}, and
\begin{enumerate}
\item Let $\beta\in (\frac{\Delta}{\chi_i},\frac{\Delta}{\chi_{i+1}})\cap \mathbb{Q}$, and let $\alpha:=(\Delta-\chi_{i+1}\beta)\in (0,\Delta\cdot(1-\frac{\chi_{i+1}}{\chi_i}))$ be the {\em scaling parameters}.
 \item Define the {\em $i+1$-th $(\Delta,\beta,n,\epsilon
)$-puck} at $x$, $$P_{n,\epsilon}(x):=f^{-\lceil n\beta\rceil}[B^T(f^{\lceil n\beta\rceil}(x),e^{-\alpha n-\epsilon n})],$$
for all $n
\in \mathbb{N}$. To ease notation, we refer to $(\Delta,\beta,n,\epsilon
)$-pucks as simply {\em $(n,\epsilon)$-pucks}. 
	\item Set the {\em $i+1$-th intermediate entropy} of $\mu$ at $x$:
$$h_{i+1}^I(x)=h^I_{i+1}(x,\Delta,\beta):=
 \lim_{\epsilon\to0} \limsup_{n\to\infty} \frac{-1}{n}\log\mu_{\xi_{i+1}(x)}(P_{n,\epsilon}(x))
.$$
\item The {\em lower $i+1$-th intermediate entropy} of $\mu$ at $x$ is defined by
$$\underline{h}^I_{i+1}(x)= \underline{h}^I_{i+1}(x,\Delta,\beta):=\lim_{\epsilon\to0} \liminf_{n\to\infty}\frac{-1}{n}\log\mu_{\xi_{i+1}(x)}(P_{n,\epsilon}(x)).$$
\end{enumerate}
\end{definition}

\noindent\textbf{Remarks:} \begin{enumerate}
\item The limit over $\epsilon$ in the definition of the intermediate entropy exists (similarly for the lower intermediate entropy), since it is monotone in $\epsilon$.
\item Note that it is clear from definition that the lower intermediate entropy is always non-negative (and so does the intermediate entropy).
\item The transverse metric on $\xi_{i+1}$ is canonical in the sense that all transverse metric are equivalent due to the Lipschitz property of the $\xi_i$ foliations, hence the intermediate entropy and the lower intermediate entropy do not depend on the metric $d^T_{i+1}(\cdot,\cdot)$.
\item The notion of intermediate entropy (and similarly lower intermediate entropy) extends naturally to non-ergodic measures. By the point-wise ergodic theorem each unstable leaf carries typical points of at most one ergodic component, and so the conditional measures are always computed for an ergodic component. 
\item  The definition extends trivially to the case where $\mu$ admits no positive Lyapunov exponents: The conditional measures are Dirac delta measures, hence $h^I_{i+1}(x)=0$ $\mu$-a.e; Therefore from here onwards we assume that $\mu$ admits a positive Lyapunov exponent a.e.
\item The definition of the $i+1$-th intermediate entropy extends to the case $i=0$ where $\xi_0$ is the partition into points, and so the transverse metric is simply the metric of $\xi_1$.
\item The assumption $\chi^u>1$ does not make us lose generality, since if $\chi_u\in (0,1)$ we may always consider $f^{\lceil\frac{2}{\chi_u}\rceil}\in \mathrm{Diff}^{1+\gamma'}(M)$ which preserves $\mu$ and $\chi_u(\mu, f^{\lceil\frac{2}{\chi_u}\rceil})\geq2>1 $.
\item Note, by the choice of $\alpha$ and $\beta$, when $\epsilon>0$ is small enough so $\chi_i\cdot\beta>\Delta+\epsilon$,
\begin{align}\label{puckWidth}
 	e^{-n\alpha-n\epsilon}\cdot e^{-n\beta\chi_{i+1}}=&e^{-\Delta n-n\epsilon}\gg e^{-n\beta\chi_i}. 	
\end{align}
Therefore the pucks are ``long" in the $\chi_{i+1}$-direction, but are very ``thin" in  other directions. 
\end{enumerate}

\section{Covers, differentiation, and the  intermediate entropy}\label{interEnt}

\subsection{Covering and differentiation lemmas}

In this section we prove some useful geometric measure theoretic properties of pucks, and give a formula to the intermediate entropy. This will be the underlying foundations which allows us to later on compute the tubular dimension (\textsection \ref{totallyTubular}), prove the leaf-wise asymptotic local product structure (\textsection \ref{AympLocl5}), and bound the entropy gaps by volume growth (\textsection \ref{entGap}).

\begin{lemma}[Puck covers]\label{BesiOptimum}
There exists a constant $C_d$ which depends only on the dimension $d$, s.t for a.e $x$, for every $n\in\mathbb N$
, and for every $\epsilon>0$, every set $A\subseteq \xi_{i+1}(x)$ can be covered by pucks centered at $A$, $\mathcal{P}_{n}$ s.t for all $P\in \mathcal{P}_{n}$, $\#\{P'\in \mathcal{P}_n:P\cap P'\neq \varnothing\}\leq C_d$.
\end{lemma}
\begin{proof}
Assume w.l.o.g that $n\beta\in \mathbb{N}$. Notice that pucks refine the partition $f^{-n\beta}[\xi_{i+1}]$, therefore it is enough to cover with a bounded multiplicity $A\cap f^{-n\beta}[\xi_{i+1}(x_0)]$ for $x_0\in A$. Then we can cover by $f^{n\beta}[P_{n,\epsilon}(\cdot)]$ the set $f^{n\beta}[A]\subseteq \xi_{i+1}(f^{n\beta}(x_0))$. Note, that $f^{n\beta}[P_{n,\epsilon}(\cdot)]=B^T(f^{n\beta}(\cdot),e^{-\alpha n-n\epsilon})$ by definition, hence we simply cover by transverse balls, which admit a Besicovitch cover with a multiplicity bounded by $C_d$.
\end{proof}

\begin{lemma}[Puck differentiation]\label{HitDiff}
Let $A\in\mathcal{B}$, then for all $\epsilon>0$, $\mu$-a.e $x\in A$:
$$\lim_{\epsilon\to 0}\limsup_{n\to\infty} \frac{-1}{n}\log\frac{\mu_{\xi_{i+1}(x)}(P_{n,\epsilon}(x)\cap A)}{\mu_{\xi_{i+1}(x)}(P_{n,\epsilon}(x))}=0.$$
\end{lemma}
\begin{proof}
	The proof is similar to \cite[Lemma~2.3]{NLE}, where Lemma \ref{BesiOptimum} is used in place of \cite[Lemma~2.2]{NLE}.
\end{proof}

\subsection{The intermediate entropy}

\begin{theorem}\label{puckEnt} Let $1\leq i\leq u-1$, then for $\mu$-a.e $x$,
	$$h^I_{i+1}(x)= \underline{h}^I_{i+1}(x)=\beta h_{i+1}+\alpha\frac{h_{i+1}-h_i}{\chi_{i+1}}.$$
\end{theorem}
\begin{proof} Let $n\in \mathbb{N}$ and let $\epsilon>0$, and assume w.l.o.g that $n\beta\in \mathbb{N}$. The idea of the proof is to compute $\mu_{\xi_{i+1}(f^{n\beta}(x))}(B^T(f^{n\beta}(x), e^{-\alpha n-n\epsilon}))$, and then estimate the change of conditional measure when pulling backwards by $f^{-n\beta}$.

Assume first that for all $\delta>0$ there exists $\epsilon_\delta>0$ small enough, s.t for all $\epsilon\in (0,\epsilon_\delta)$, there exists $n_{\epsilon}$ s.t $\forall n\geq n_\delta$, 
\begin{equation}\label{tubeEst}
	\mu_{\xi_{i+1}(f^{n\beta}(x))}(B^T(f^{n\beta}(x), e^{-\alpha n-n\epsilon}))=e^{-n\alpha\frac{h_{i+1}-h_i}{\chi_{i+1}}\pm \delta n}.
\end{equation}
Under this assumption, we continue to estimate $\mu_{\xi_{i+1}(x)}(P_{n,\epsilon}(x))$: Write $Q_k:=f^k[P_{n,\epsilon}(x)]$ and $\mu_k:=\mu_{\xi_{i+1}(f^k(x))}$, then
\begin{align}\label{TeleProd}
\mu_{\xi_{i+1}(x)}(P_{n,\epsilon}(x))=&\prod_{k=0}^{n\beta-1}\frac{\mu_k(Q_k)}{\mu_{k+1}(Q_{k+1})}\cdot \mu_{n\beta}(Q_{n\beta}).
\end{align}
Note, $f^{-1}[Q_{k+1}]=Q_k$, then by the invariance of $\mu$ and the uniqueness of conditional measures:
$$\frac{\mu_k(Q_k)}{\mu_{k+1}(Q_{k+1})} =\mu_{k}(f^{-1}[\xi_{i+1}(f^{k+1}(x))])=:e^{-I_{i+1}(x)}.$$
By the point-wise ergodic theorem, for $\mu$-a.e $x$, $$\lim_{n\beta\to\infty}\frac{1}{n\beta}\sum_{k=0}^{n\beta-1}I_{i+1}\circ f^k(x)=\int I_{i+1}d\mu.$$
By \cite[\textsection~9.2, \textsection~9.3]{LedrappierYoungII}, the r.h.s equals $h_{i+1}$, and so we are only left to prove \eqref{tubeEst}.

By \cite[\textsection~11.4]{LedrappierYoungII}, for $\mu$-a.e $y$, $$\lim_{n\to\infty}\frac{-1}{n}\log\mu_{\xi_{i+1}(y)}(B^T(y,e^{-\alpha n-n\epsilon}))=\frac{h_{i+1}-h_i}{\chi_{i+1}}(\alpha+\epsilon).$$

Let $\delta>0$ small, and let $n_\delta'\in \mathbb{N}$ large enough so $\mu(\Omega_\delta)\geq 1-\delta^3$ where
\begin{align*}
	\Omega_\delta:=\Big\{y: \forall n\geq \delta^2n_\delta',& \mu_{\xi_{i+1}(y)}(B^T(y,e^{-\alpha n-n\epsilon}))= e^{-n\frac{h_{i+1}-h_i}{\chi_{i+1}}(\alpha+\epsilon)\pm \delta^2 n}, \\
	&\frac{1}{n}\sum_{\ell=0}^{n-1}I_{i+1}(f^{-\ell}(y))=h_{i+1}\pm \delta^2\Big\}.
\end{align*}
By the point-wise ergodic theorem, for $\mu$-a.e $x$ there exists $n_\delta\geq n_\delta'$ s.t $\forall n\geq n_\delta$, $\exists j\in [n(1+\delta^2),n(1+2\delta^2)]$ s.t $f^j(x)\in \Omega_\delta$, thus 
\begin{align*}
	\mu_{n\beta}(Q_{n\beta})\geq &\mu_n(f^{-(j-n)}[B^T(f^{j}(x),e^{-\alpha n-\epsilon n})])\\
	\geq &e^{-n\frac{h_{i+1}-h_i}{\chi_{i+1}}(\alpha+\epsilon)\pm \delta^2 n} \cdot \exp\left(-(j-n)\frac{1}{j-n}\sum_{\ell=n}^{j-1}I_{i+1}(f^{-\ell})(f^j(x))\right).
\end{align*}
Since $2\delta^2n\geq j-n\geq \delta^2 n$, by the definition of $\Omega_\delta$, for all $\delta>0$ small enough, we get $\mu_{n\beta}(Q_{n\beta})\geq e^{-n\alpha\frac{h_{i+1}-h_i}{\chi_{i+1}}- \delta n} $. The upper bound is shown similarly.
\end{proof}

\noindent\textbf{Remark:} Note, the formula for the intermediate entropy can be also written as $h^I_{i+1}(x)=\beta h_i+\Delta \frac{h_{i+1}-h_i}{\chi_{i+1}}$. This alludes to the asymptotic local product structure which we prove in \textsection \ref{AympLocl5}, since it implies:
$$\mu_{\xi_{i+1}(x)}(P_{n,\epsilon}(x))\approx \mu_{\xi_i(x)}(P_{n,\epsilon}(x))\cdot \mu_{\xi_{i+1}(x)}(B^T(x,e^{-\Delta n})).$$

\begin{lemma}\label{puckEnt1}
	For $\mu$-a.e $x$,
	$$h^I_{1}(x)= \underline{h}^I_{1}(x)=\frac{h_{1}}{\chi_1}(\alpha+\chi_1\beta).$$
\end{lemma}
\begin{proof}
The definition of pucks indeed extends to the case of $i=0$ (see the sixth itemized remark after Definition \ref{Defs2}): $P_{n,\epsilon}(x)=f^{-n\beta}[B_{\xi_1(x)}(f^{n\beta}(x),e^{-n\alpha-n\epsilon})]$. Therefore, for all $\delta>0$, for all $\epsilon>0$ small enough, for $\mu$-a.e $x$, for all $n$ large enough so $n\beta\in \mathbb{N}$,
 $$B_{\xi_1(x)}(x,e^{-\Delta n-\delta n}) \subseteq P_{n,\epsilon}(x)\subseteq B_{\xi_1(x)}(x,e^{-\Delta n+\delta n}).$$ 
Thus $h^I_1(x)=\Delta \frac{h_1}{\chi_1}=\frac{\alpha+\chi_1\beta}{\chi_1}h_1$ $\mu$-a.e.
\end{proof}

\medskip
\noindent\textbf{Remark:} The formula in Lemma \ref{puckEnt1} coincides with the formula in Theorem \ref{puckEnt} when substituting $i=0$ and $h_0=0$.

\section{Tubes and tubular dimension}\label{totallyTubular}

For $i=2$, the action of differential in the $\xi_1$ leaves is almost conformal- in the sense that around good Pesin points, the differential acts with the maximal Lyapunov exponent, up to a small exponential error. 

However, for $i>2$, this is no longer the case, as the action of the differential acts with different exponents in different Oseledec directions, which breaks the conformality. In turn, this gives the pucks a ``distorted" shape in the $\xi_i$ direction. 

The purpose of this section is to treat this issue, and define a proper object with a nice geometric description, whose measure we can estimate. This is done by saturating the pucks into {\em tubes}. Recall Definition \ref{Defs1}.

\begin{definition}[Tubes]\label{tubes}
Let $1\leq i\leq u-1$.
\begin{enumerate}
	\item We restrict the range of values of the scaling parameter $\beta$ further-
let $$\beta_{i+1}\in \left(\frac{\Delta}{\chi_i},\min\left\{\frac{\Delta}{\chi_{i}\cdot(1-\theta)},\frac{\Delta}{\chi_{i+1}}\right\}\right)\cap\mathbb{Q},$$ where $\theta>0$ is the 
H\"older exponent of the $i+1$-th Oseledec direction on $\Lambda^{(\underline\chi,\tau)}$, whenever $0<\tau<\frac{1}{4}\min_{i\neq j}\{|\chi_i-\chi_j|,|\chi_u|\}$. $\beta_{i+1}$ is called the {\em tubular scale}. Since $\Delta$ is fixed, we write $\beta_{i+1}$ to denote that it lies in the restricted range of values. 
	\item Fix a Pesin block $\Lambda:=\Lambda^{(\underline\chi,\tau)}_\ell$, and let $x\in \Lambda$. Then the $i+1$-th {\em $(n,\epsilon,\Lambda)$-pre-tube} at $x$ is defined by $$\check{T}_{n,\epsilon}(x):=\bigcup\left\{P_{n,\epsilon}(y):y\in B_{\xi_i(x)}(x,e^{-n\chi_i\beta_{i+1}})\cap \Lambda\right\}.$$
Given $\Delta>0$, the $i+1$-th {\em $(\Delta,\beta_{i+1},n,\epsilon)$- tube} at $x$ is defined by
$$T_{n,\epsilon}(x):=B^T(x,e^{-\Delta n})\cap \{y:|\pi_i(y)-\pi_i(x)|\leq e^{-n\chi_i\beta_{i+1}+\sqrt{\epsilon} n}\}.$$
To ease notation we refer to $(\Delta,\beta_{i+1},n,\epsilon)$-tubes as {\em $(n,\epsilon)$-tubes} or {\em tubes}.
\end{enumerate}
\end{definition}

\noindent\textbf{Remark:} Note, $\theta$ is well-defined and depends only on $\underline\chi$, by \cite[Appendix~A]{BrinHolderContFoliations}. In addition, note that 
we have 
$\max\{\chi_{i+1}\beta_{i+1}, \chi_i\beta_{i+1}(1-\theta)\}<\Delta<\chi_i\beta_{i+1}$.

\subsection{Tubular covers and differentiation}

\begin{lemma}[Tubular covers]\label{tCover}
For $\mu$-a.e $x$, for every $A\subseteq \xi_{i+1}(x)$, $A$ can be covered by tubes in a Besicovitch manner, with a multiplicity bounded by $\widetilde{C}_d$, where $\widetilde{C}_d$ is a constant depending on the manifold and the partitions $\xi_i,\xi_{i+1}$.
\end{lemma}
\begin{proof}
We first cover $A_i:=\pi_{i+1}^{-1}[A]$ by a specific cover of transverse balls: Take $x_1\in A_i$. Given $\{x_1,\ldots,x_n\}$, take $x_{n+1}\in A\setminus \bigcup_{j\leq n} B^T(x_j, \frac{1}{2}e^{-\Delta n})$. The 
process continues as long as $A_i$ is not covered. However, $\forall j\neq j'$, $B^T(x_j, \frac{1}{6}e^{-\Delta n}) \cap B^T(x_{j'}, \frac{1}{6}e^{-\Delta n}) $, therefore the process stops at a finite time $N$, as the quotient space $\xi_{i+1}/\xi_i$ admits a metric with finite volume. Let $\Vol_T$ be the volume on the quotient space, then the cover $\mathcal{C}_T$ has multiplicity bounded by $\frac{\max \Vol_T(B^T(\cdot,\frac{3}{2} e^{-\Delta n}))}{\min\Vol_T(B^T(\cdot, \frac{1}{6} e^{-\Delta n}))}\leq \frac{\max \Vol_T(B^T(\cdot,\frac{9}{2} e^{-\Delta n}))}{\min\Vol_T(B^T(\cdot, \frac{1}{6} e^{-\Delta n}))}\equiv \widetilde{C}_{i+1}$.

For every $B\in \mathcal{C}_T$, cover $A\cap B$ in a Besicovitch cover by tubes, with multiplicity bounded by $C_d$. Denote this cover by $\mathcal{C}^B$. This is possible since each tube extends across $B$, and it is enough to cover $\pi_i[B\cap A]$ in $\xi_i$ by balls. 

We claim that the cover $\mathcal{C}:=\bigcup\{\mathcal{C}^B:B\in\mathcal{C}_T\}$ has multiplicity bounded by $\widetilde{C}_{i+1}\cdot C_d\equiv \widetilde{C}_d$. For every $B\in \mathcal{C}_T$, $\bigcup \mathcal{C}^B\subseteq \frac{3}{2}B$. However, the number of $B'\in\mathcal{C}^T$ s.t $\frac{3}{2}B\cap \frac{3}{2}B'\neq \varnothing $ is bounded by $\frac{\Vol_T(\frac{9}{2}B)}{\min \Vol(\frac{1}{6}B')}\leq \widetilde{C}_{i+1}$. Inside each $B'$, the cover has multiplicity bounded by $C_d$, and we're done.
\end{proof}

\begin{cor}[Tubular differentiation]\label{tDiff}
	Let $A\in\mathcal{B}$, then for all $\epsilon>0$, $\mu$-a.e $x\in A$:
$$\lim_{\epsilon\to 0}\limsup_{n\to\infty} \frac{-1}{n}\log\frac{\mu_{\xi_{i+1}(x)}(T_{n,\epsilon}(x)\cap A)}{\mu_{\xi_{i+1}(x)}(T_{n,\epsilon}(x))}=0.$$
\end{cor}
\begin{proof}
	The proof is similar to \cite[Lemma~2.3]{NLE}, where Lemma \ref{tCover} is used in place of \cite[Lemma~2.2]{NLE}.
\end{proof}

\subsection{Tubular dimension}

\begin{prop}\label{PreTAndT}
	Let $\epsilon>0$ small, and let $\Lambda=\Lambda^{(\underline\chi,\tau)}$ be a Pesin block with $0<\tau<\epsilon$. Then for all $n$ large enough, for $\mu$-a.e $x\in \Lambda$, 
	$$T_{n,\epsilon}(x) \supseteq \check{T}_{n,\epsilon}(x).$$
\end{prop}
\begin{proof}
	For all $\epsilon>0$, for every $y\in \xi_i(x)\cap \Lambda$, $P_{n,\epsilon}(y)\subseteq B^T(x,e^{-n\chi\Delta })$ for all $n$ large enough w.r.t $\Lambda$. Therefore, we wish to show that $\forall z\in P_{n,\epsilon}(y)$, $|\pi_i(z)-\pi_i(x)|\leq e^{-n\chi_i\beta_{i+1}+\sqrt{\epsilon} n}$.
	
Note, given $y_1,y_2\in B_{\xi_i(x)}(x,e^{-n\beta _{i+1}\chi_i})$, and given $z_1\in P_{n,\epsilon}(y_1)$ and $z_2\in P_{n,\epsilon}(y_2)$, the disposition $|\pi_i(z_1)-\pi_i(z_2)|$ is bounded by:
\begin{align*}
|\pi_i(z_1)-\pi_i(z_2)|\leq 2e^{-n\beta _{i+1}\chi_i}+2\cdot (e^{-n\Delta}\cdot \sin\sphericalangle (E_{i+1}(y_1), E_{i+1}(y_1))+e^{-n\beta _{i+1}\chi_i+n2\epsilon}).
\end{align*}
The angle is bounded by $\sphericalangle (E_{i+1}(y_1), E_{i+1}(y_1))\leq C_\mathrm{H\ddot{o}l}(\Lambda)\cdot e^{-n\beta _{i+1}\chi_i \theta}$. Since $\beta_{i+1}<\frac{\Delta}{\chi_i(1-\theta)}$, we have
$$e^{-n\Delta} \cdot e^{-n\beta _{i+1}\chi_i \theta}\leq e^{-n\beta _{i+1} \chi_i
}.$$
Thus, for all $n$ large enough, $T_{n,\epsilon}(x)\supseteq\check{T}_{n,\epsilon}(x)$.
\end{proof}

\noindent\textbf{Remark:} Proposition \ref{PreTAndT} is the key place where we need to use the choice of parameters $\Delta$, and consequently $\beta_{i+1}$. They are necessary to compensate for the fact that the $E_{i+1}(\cdot)$ distribution is merely H\"older continuous on Pesin blocks and not Lipschitz. The key is to ``shorten" a Bowen ball enough into a puck, so the H\"older regularity is sufficient. The lack of conformality (i.e $\chi_i>\chi_{i+1}$) is what allows us to make sure that the tubes remain still ``elongated" in shape, with a uniform exponential gap on their dimensions.

\begin{definition}[Tubular dimension]\label{tEnt}
The {\em $i+1$-th tubular dimension} at a point $x$ is defined by
$$\rho^T_{i+1}(x)=\rho^T(x,\Delta,\beta_{i+1}):=\lim_{\epsilon\to0}\limsup_{n\to\infty}\frac{-1}{n}\log\mu_{\xi_{i+1}(x)}(T_{n,\epsilon}(x)).$$
The {\em $i+1$-th lower tubular dimension} at a point $x$ is defined by
$$\underline{\rho}^T_{i+1}(x)=\underline{\rho}^T(x, \Delta,\beta_{i+1}):=\lim_{\epsilon\to0}\liminf_{n\to\infty}\frac{-1}{n}\log\mu_{\xi_{i+1}(x)}(T_{n,\epsilon}(x)).$$
\end{definition}

\noindent\textbf{Remark:} To understand the role of the scaling parameters for the tubular dimension, allow us illustrate an example. Assume that $f$ is a linear map with orthogonal eigen-directions. Then inside each $W^2$ leaf, we can consider a tube $T$ around a point $x$, with sides parallel to $E_1$ and $E_2$, with dimensions $e^{-n\chi_1}$ and $e^{-\Delta n}$, respectively. The window parameter $\Delta$ can vary between $\chi_2$ and $\chi_1$,\footnote{In this example $\beta$ is fixed with the value $1$, hence $\frac{\Delta}{\chi_1}<1<\frac{\Delta}{\chi_2}\Rightarrow \chi_2<\Delta<\chi_1$.} yielding a Bowen ball and a square ($\|\cdot\|_\infty$-ball), respectively. For both extreme values, the measure of $T$ is known by either the local entropy estimates, or by the point-wise dimension estimates. The idea behind tubular dimension is to be able to estimate the measure of the tube for the rest of the range of possible values of $\Delta$. In generality, the mere H\"older continuity of the $E_{2}$ direction restricts further the range of possible dimensions whose measure we can estimate. Moreover, when considering higher order intermediate foliations, the action of the differential on the $E_i$ direction, $i>1$, is no longer conformal, and so we need to saturate pucks into tubes.

\begin{theorem}\label{Grail5}
	For every $1\leq i\leq u-1$, for $\mu$-a.e $x$, $$\underline{\rho}^T_{i+1}(x) = \rho^T_{i+1}(x)=\Delta\cdot \frac{h_{i+1}-h_i}{\chi_{i+1}}+\beta _{i+1}\chi_i d_i.$$
\end{theorem}
\begin{proof}
	We begin by showing $\rho_{i+1}^T\leq \Delta\cdot (h_{i+1}-h_i)+\beta _{i+1}\chi_i d_i$ $\mu$-a.e. By Theorem \ref{puckEnt} (recall also the succeeding remark), for all $\delta>0$, for $\mu$-a.e $x$, there exists $\epsilon_\delta>0$ s.t for all $\epsilon\in (0,\epsilon_\delta)$, for all $n$ large enough w.r.t $\epsilon$, $\mu_{\xi_{i+1}(x)}(P_{n,\epsilon}(x))\geq e^{-n\beta _{i+1} h_i}\cdot e^{-n\Delta\frac{h_{i+1}-h_i}{\chi_{i+1}}}e^{-\delta n}$. Let $\epsilon\in (0,\delta)$, and let $\tau\in (0,\epsilon)$, $\ell\in\mathbb{N}$, and $n_\epsilon$ s.t $\mu(\Omega_\delta)\geq 1-\epsilon$, where
\begin{align*}
\Omega_\delta:=\{x\in \Lambda^{(\underline\chi,\tau)}_\ell:\forall n\geq n_\epsilon,\text{ }&\mu_{\xi_{i+1}(x)}(P_{n,\epsilon}(x))\geq e^{-n\beta _{i+1} h_i}\cdot e^{-n\Delta\frac{h_{i+1}-h_i}{\chi_{i+1}}}e^{-\delta n},\\
&\mu_{\xi_i(x)}(P_{n,\epsilon}(x))= e^{-nh_i\beta _{i+1}\pm \delta n}\}.
\end{align*}

Let $x$ be a $\mu_{\xi_i(x)}$-density point of $\Omega_\delta$, then there exists $n_x\geq n_\epsilon$ s.t for $n\geq  n_x$, $$\frac{\mu_{\xi_i(x)}(B_{\xi_{i}(x)}(x,e^{-\chi_i \beta _{i+1} n})\cap \Omega_\delta)}{\mu_{\xi_i(x)}(B_{\xi_{i}(x)}(x,e^{-\chi_i \beta _{i+1} n}))}\geq 1-\delta.$$
Moreover, for $\mu$-a.e $x\in \Omega_\delta$ exists $n_x'\geq n_x$ s.t $\mu_{\xi_i(x)}(B_{\xi_{i}(x)}(x,e^{-\chi_i \beta _{i+1} n})) \geq e^{-n\chi_i\beta _{i+1} d_i-\delta n}$.

We may cover $\Omega_\delta\cap B_{\xi_{i}(x)}(x,e^{-\chi_i \beta _{i+1} n}) $ by Lemma \ref{BesiOptimum}, and there must be at least \\$\frac{(1-\delta) e^{-n\chi_i\beta _{i+1} d_i-\delta n}}{e^{-nh_i\beta _{i+1} +\delta n}}$-many elements in this cover. Thus, by Proposition \ref{PreTAndT},
\begin{align*}
	\mu_{\xi_{i+1}(x)}(T_{n,\epsilon}(x))\geq& \frac{1}{C_d}\cdot \frac{(1-\delta) e^{-n\chi_i\beta _{i+1} d_i-\delta n}}{e^{-nh_i\beta _{i+1} +\delta n}} \cdot e^{-n\beta _{i+1} h_i}\cdot e^{-n\Delta\frac{h_{i+1}-h_i}{\chi_{i+1}}}e^{-\delta n} \\
	\geq &e^{-n(\Delta\cdot \frac{h_{i+1}-h_i}{\chi_{i+1}}+\beta _{i+1}\chi_i d_i)}e^{-4\delta n}.
\end{align*}
	Since $\delta>0$ was arbitrarily small, this concludes the upper bound of $\rho_{i+1}^T$.

\medskip
We continue to bound from below the lower tubular dimension: Let $\delta\in (0,1)$, and let $h\geq 0$ s.t $\mu([\underline{\rho}_{i+1}^T=\rho\pm \frac{\delta}{3}])>0$. There exists $\epsilon\in (0,\delta^3)$ s.t $\mu(A)>0$ where 
\begin{align*}
	A:=\{x\in \Lambda^{(\underline\chi,\epsilon^2)}:\text{ }&\liminf\frac{-1}{n}\log\mu_{\xi_{i+1}(x)}(T_{n,\epsilon}(x))=\rho\pm\frac{2\delta}{3},\\
	\forall n\geq n_0,\text{ }& \mu_{\xi_{i+1}(x)}(B^T(x,e^{-\Delta n+\delta^2 n}))\leq e^{-\Delta\frac{h_{i+1}-h_i}{\chi_{i+1}}n+\delta n}\\
	& \mu_{\xi_{i}(x)}(B_{\xi_{i}(x)}(x,e^{-n\beta _{i+1}\chi_i-\Delta^2 n}))\leq e^{-\beta _{i+1}\chi_i d_{i}n+\delta n} \}.
\end{align*}

Let $A_n:=\{x\in A: \frac{-1}{n}\log\mu_{\xi_{i+1}(x)}(T_{n,\epsilon}(x))=\rho\pm\delta\}$, then by the Borel-Cantelli lemma, there exist infinitely many $n$ s.t $\mu(A_n)\geq e^{-\epsilon n}$.

Let $x_n$ s.t $\mu_{\xi_{i}(x_n)}(A_n)\geq e^{-\epsilon n}$. Cover $A_n\cap \xi_{i}(x_n)$ by tubes as in Lemma \ref{tCover}. Then the number of tubes is bigger than $e^{-n\epsilon}\cdot e^{\beta _{i+1}\chi_i d_{i}n-\delta n}  $. On the other hand, the union of the tubes is contained in the transverse ball $B^T(x,e^{-\Delta n+\delta^2 n}) $, therefore,
$$e^{-\Delta\frac{h_{i+1}-h_i}{\chi_{i+1}}n+\delta n} \geq e^{-n\epsilon}\cdot \frac{1}{\widetilde{C}_d} e^{\beta _{i+1}\chi_i d_{i}n-\delta n} \cdot e^{-\rho n-\delta n}.$$
Since $\delta>0$ was arbitrary small, $\epsilon<\delta$, and $n$ can be arbitrarily large, we are done.
\end{proof}

\noindent\textbf{Remark:} For the case of $i=0$, one can think of a tube as the collection of pucks centered inside a ball in $\xi_0(x)$. However, $\xi_0(x)=\{x\}$, therefore the tube is simply a puck and the tubular dimension is the intermediate entropy; And in this case we refer to Lemma \ref{puckEnt1}. Indeed, when substituting $h_0=d_0=0$ in the formula in Theorem \ref{Grail5}, we get the formula from Lemma \ref{puckEnt1}.

\section{Leaf-wise asymptotic local product structure}\label{AympLocl5}

The notion of tubes and tubular dimension were designed so we can compute the measure of an object which admits a nice geometric description, and which satisfies the property that when intersected with a transverse ball, the intersection is a ball. This allows us to employ all three estimates simultaneously: the measure of balls (the dimension of the conditional measures), the measure of transverse balls (the transverse dimension), and finally the measure of tubes (tubular dimension). The beauty of these quantities is that they satisfy a multiplicative relationship, which allows us to conclude the asymptotic local product structure of conditional measures. 

\begin{theorem}[Leaf-wise asymptotic local product structure]
For all $1\leq i\leq u-1$, for all $\Delta>0$, for all $\delta>0$, for all $\epsilon>0$ small enough, there exists a set $K_\delta$ s.t $\mu(K_\delta)\geq 1-\delta$ and for $\mu$-a.e $x\in K_\delta$, for all $n
\in\mathbb{N}$ large enough,
\begin{enumerate}
	\item $\frac{\mu _{\xi_{i+1}(x)}(B_{\xi_{i+1}(x)}(x,e^{-n\Delta-n\epsilon})\cap K_\delta)}{\mu _{\xi_{i+1}(x)}(B_{\xi_{i+1}(x)}(x,e^{-n\Delta-n\epsilon}))}\geq 1-\delta$,
	\item there exist a cover of $B_{\xi_{i+1}(x)}(x,e^{-n\Delta-n\epsilon})\cap K_\delta$, $\mathcal{T}_{n,\epsilon}$, by $i+1$-th $(n,\epsilon)$-tubes centered at $K_\delta$, and a cover $\mathcal{B}^T_{n,\epsilon}$ by transvese balls $B^T(\cdot,e^{-n\chi_{i}\beta _{i+1}-n\epsilon})$ centered at $K_\delta$, both with multiplicity bounded by $\widetilde{C}_d$,
	\item set $N^\mathcal{T}_n:=\# \mathcal{T}_{n,\epsilon}$, $N^\mathcal{B}_{n}:=\# \mathcal{B}^T_{n ,\epsilon} $, and $N_n:=\#\{B\in \mathcal{T}_{n,\epsilon}\vee \mathcal{B}^T_{n ,\epsilon}: B\cap K_\delta\neq\varnothing\}$, then
\begin{equation}\label{leafAsympLocProd}
	\frac{1}{n}\left|\log\frac{N_n^\mathcal{T}\cdot N_n^\mathcal{B}}{N_n}\right|\leq\delta.
\end{equation}
\end{enumerate}
\end{theorem}
Note that the definition of $\beta_{i+1}$ (Definition \ref{tubes}) guarantees that $\Delta=\chi_i\beta _{i+1}\cdot(1-\theta)<\chi_i\beta _{i+1} $ thus the tubes and the transverse balls are much ``thinner" than the dimensions of the ball $B_{\xi_{i+1}(x)}(x,e^{-n\Delta-n\epsilon})$.
\begin{proof}
Let $\delta\in (0,1)$ small, and let $\epsilon\in (0,\delta^8)$ small enough so for every $\epsilon\in (0,\epsilon_\delta)$ there exists $n_\epsilon\in\mathbb{N}$ s.t $\mu(K_\delta)\geq 1-\delta$, where 
\begin{align*}
K_\delta:=\Big\{x\in \Lambda^{(\underline\chi,\epsilon^3)}_{\ell_\delta}: \forall n\geq n_\epsilon,\text{ }
& \mu_{\xi_{i+1}(x)}(B^T(x,e^{-n\chi_{i}\beta _{i+1}-n\epsilon}))=e^{-n((\chi_i\beta _{i+1} +\epsilon)\frac{h_{i+1}-h_i}{\chi_{i+1}}\pm \delta^2)},\\ 
&\mu _{\xi_{i+1}(x)}(B_{\xi_{i+1}(x)}(x,e^{-n\Delta\pm n\epsilon}))=e^{-n(\Delta+\epsilon)d_{i+1}\pm\delta^2n},\\
& \mu_{\xi_{i+1}(x)}(B^T(x,e^{-n\Delta+n\epsilon}))=e^{-n((\Delta-\epsilon)\frac{h_{i+1}-h_i}{\chi_{i+1}}\pm \delta^2)},\\
&\mu_{\xi_{i}(x)}(T_{n,\epsilon}(x))= e^{-n(\rho_{i+1}^T\pm\delta^2)},\\
&\mu_{\xi_{i}(x)}(B_{\xi_{i+1}(x)}(x,e^{-n\chi_i\beta _{i+1} +\sqrt\epsilon n}))= e^{-n(\chi_i\beta _{i+1} d_{i+1}\pm\delta^2)}
\Big\}.
\end{align*}
This is possible by Theorem \ref{Grail5} and by 
\cite
{LedrappierYoungII}.

Let $x$ be a $\mu_{\xi_{i+1}(x)}$-Lebesgue density point of $K_\delta$ and a $\mu_{\xi_{i+1}(x)}\circ \pi_i^{-1}$-Lebesgue density point of $\xi_i[K_\delta]$. Let $n_x\geq n_\epsilon$ s.t for all $n\geq n_x$,
\begin{align*}
\frac{\mu _{\xi_{i+1}(x)}(B_{\xi_{i+1}(x)}(x,e^{-n\Delta-n\epsilon})\cap K_\delta)}{\mu _{\xi_{i+1}(x)}(B_{\xi_{i+1}(x)}(x,e^{-n\Delta-n\epsilon}))}\geq & 1-\delta,\\
\frac{\mu _{\xi_{i+1}(x)}(B^T(x,e^{-n\Delta+n\epsilon})\cap \xi_i[K_\delta])}{\mu _{\xi_{i+1}(x)}(B^T(x,e^{-n\Delta+n\epsilon}))}\geq & 1-\delta.
\end{align*}

By Lemma \ref{tCover}, $K_\delta\cap B_{\xi_{i+1}(x)}(x,e^{-n\Delta-n\epsilon}) $ can be covered by a cover of $(n,\epsilon)$-tubes centered at $K_\delta\cap B_{\xi_{i+1}(x)}(x,e^{-n\Delta-n\epsilon})$, $\mathcal{T}_{n,\epsilon}$, and by a cover of $B^T(\cdot,e^{-n(\chi_i\beta _{i+1} +\epsilon)})$ transverse balls centered at $K_\delta\cap B_{\xi_{i+1}(x)}(x,e^{-n\Delta-n\epsilon})$, $\mathcal{B}^T_{n,\epsilon}$, both with multiplicity bounded by $\widetilde{C}_d$. 

\medskip
In order to bound $N_n^\mathcal{T}$ from above, note that $\bigcup \mathcal{T}_{n,\epsilon}\subseteq B_{\xi_{i+1}(x)}(x,e^{-n\Delta+ n\epsilon})$, and so for all $n$ large enough w.r.t $\delta$,
\begin{equation}\label{NTbound}
	N_n^\mathcal{T}\leq \widetilde{C}_d\frac{e^{-n(\Delta+\epsilon)d_{i+1}+\delta^2n}}{e^{-n(\rho_{i+1}^T-\delta^2)}}\leq e^{n(\rho_{i+1}^T-\Delta d_{i+1}+3\delta^2)}.
\end{equation}
Similarly, since $ \bigcup \mathcal{T}_{n,\epsilon}\supseteq B_{\xi_{i+1}(x)}(x,e^{-n\Delta- n\epsilon})\cap K_\delta $,
\begin{equation*}
	N_n^\mathcal{T}\geq \frac{(1-\delta)e^{-n(\Delta-\epsilon)d_{i+1}-\delta^2n}}{e^{-n(\rho_{i+1}^T+\delta^2)}}\geq e^{n(\rho_{i+1}^T-\Delta d_{i+1}-3\delta^2)}.
\end{equation*}

\medskip
In order to bound $N_n^\mathcal{B}$ from above, note that $\bigcup \mathcal{B}^T_{n,\epsilon}\subseteq B^T(x,e^{-n\Delta+ n\epsilon})$, and so for all $n$ large enough,
\begin{equation}\label{NBbound}
	N_n^\mathcal{B}\leq \widetilde{C}_d\frac{e^{-n((\Delta-\epsilon)\frac{h_{i+1}-h_i}{\chi_{i+1}}+ \delta^2)}}{e^{-n((\chi_i\beta _{i+1} +\epsilon)\frac{h_{i+1}-h_i}{\chi_{i+1}}-\delta^2)}}\leq e^{n((\chi_i\beta _{i+1}-\Delta) \frac{h_{i+1}-h_i}{\chi_{i+1}} +3\delta^2)}.
\end{equation}
Similarly, since $ \bigcup \mathcal{B}^T_{n,\epsilon}\supseteq B^T(x,e^{-n\Delta- n\epsilon})\cap \xi_i[K_\delta] $,
\begin{equation*}
	N_n^\mathcal{B}\geq \frac{(1-\delta) e^{-n((\Delta-\epsilon)\frac{h_{i+1}-h_i}{\chi_{i+1}}-\delta^2)}}{e^{-n((\chi_i\beta _{i+1} +\epsilon)\frac{h_{i+1}-h_i}{\chi_{i+1}}+\delta^2)}}\geq e^{n((\chi_i\beta _{i+1}-\Delta) \frac{h_{i+1}-h_i}{\chi_{i+1}} -3\delta^2)}.
\end{equation*}

\medskip
Finally, notice that $$N_n\leq N_n^\mathcal{T}\cdot N_n^\mathcal{B},$$ since $N_n$ is bounded by the cardinality of the cover which is the refinement of $\mathcal{T}_{n,\epsilon}$ and $\mathcal{B}^T_{n,\epsilon}$. The lower bound on $N_n$ is given by the observation that for every element $B\in\mathcal{T}_{n,\epsilon}\vee\mathcal{B}^T_{n,\epsilon}$ s.t $\exists x_B\in B\cap K_\delta$, $B\subseteq B_{\xi_{i+1}(x)}(x_B,e^{-n\chi_i\beta _{i+1} +\sqrt\epsilon n})$, and so 
\begin{equation*}
	\mu_{\xi_{i+1}(x)}(B)\leq e^{-n(\chi_i\beta _{i+1} d_{i+1}-\delta^2)}.
\end{equation*}
Therefore,
\begin{equation}\label{Nbound}
	N_n\geq \frac{(1-\delta)e^{-n(\Delta-\epsilon)d_{i+1}-\delta^2n}}{e^{-n(\chi_i\beta _{i+1} d_{i+1}-\delta^2)}}\geq e^{n(\chi_i\beta _{i+1} d_{i+1}-\Delta d_{i+1}-3\delta^2)}.
\end{equation}
By Theorem \ref{Grail5}, $\rho_{i+1}^T=\chi_i\beta _{i+1} d_i+\Delta\frac{h_{i+1}-h_i}{\chi_{i+1}}$, therefore by \eqref{NTbound},\eqref{NBbound}, and \eqref{NBbound},
$$N_n\geq e^{-9\delta^2 n}\cdot N_n^\mathcal{T}\cdot N_n^\mathcal{B}.$$
Therefore, for $\delta>0$ small enough, we are done. 
\end{proof}

\section{Entropy gap and volume growth}\label{entGap}

In \textsection \ref{AympLocl5} we gave an application of the tubular dimension which was introduced in \textsection \ref{totallyTubular}. 
In this section we present a second application of the tubular dimension. We bound the difference between any two consecutive conditional entropies $h_{i+1}-h_i$ by the volume growth of disks of dimension corresponding to the Oseledec space of $\chi_{i+1}$, $E_{i+1}$.

 \begin{definition}[Newhouse \cite{NewhouseVolGrowth}]\label{UniVolGrow}
Let $1<r\in (\mathbb{R}\setminus\mathbb{N})\cup \{\infty\}$, and assume $f\in \mathrm{Diff}^r(M)$. Let $1\leq k\leq d$ be an integer, then the {\em $(k,r)$-uniform volume growth} is defined by
	$$\mathcal{V}_k^r(f):=\limsup_{n\to\infty}\frac{1}{n}\log \sup_{\eta}\Vol_{k}(f^n\circ \eta[B_{\mathbb{R}^k}(0,1)]),$$
	where the supremum is taken over all diffeomorphisms $\eta: B_{\mathbb{R}^k}(0,1)\to M$ s.t $\|\eta\|_{C^{1+\{r\}}}\leq 1$, where $\{r\}:=r-\lfloor r\rfloor$. Such disks are called {\em standard $k$-disks}.
\end{definition}

\begin{prop}\label{keyProp}
For every Pesin block $\Lambda^{(\underline\chi,\tau)}$ and a measurable set $A$ s.t $\mu(A\cap \Lambda^{(\underline\chi,\tau)})>0$, for $\mu$-a.e $x\in \Lambda ^{(\underline\chi,\tau)}\cap A $, $\exists \epsilon\in (0,\tau)$ s.t for all $n$ large enough, there exist at least $$e^{n(d_{i+1}-d_i)(\chi_i\beta_{i+1}-\Delta)-7d\tau n}\text{-many}$$ disjoint transverse balls $B^T(\cdot,e^{-n\chi_i\beta_{i+1}+3n\tau})$, s.t their intersection with $T_{n,\epsilon}(x)$ contains a $(\beta_{i+1}\chi_i,\frac{\beta_{i+1}}{1-\frac{\theta}{2}}, n,\epsilon)$-tube centered at a point in $\Lambda ^{(\underline\chi,\tau)}\cap A $.
\end{prop}
\begin{proof}
Let $n_\tau\in\mathbb{N}$ s.t $\mu(K_\tau)\geq e^{-\frac{\tau}{2}}\mu(\Lambda ^{(\underline\chi,\tau)})\cap A$, where $$K_\tau:=\{x\in \Lambda ^{(\underline\chi,\tau)}\cap A:\forall n\geq n_\tau,\mu_{\xi_{i+1}(x)}(B_{\xi_{i+1}(x)}(x,e^{-\chi_i\beta_{i+1}n}))\leq e^{-nd_{i+1}\chi_i\beta+\tau n}\}.$$

By Lemma \ref{tDiff}, for $\mu$-a.e $x\in K_\tau$, there exist $\epsilon_x\in (0,\tau^2)$ s.t $\forall \epsilon\in (0,\epsilon_x]$ $\exists n_x(\epsilon)\in \mathbb{N}$ s.t $\forall n\geq n_x$, 
\begin{equation*}\label{keyOfKey}
	\frac{\mu_{\xi_{i+1}(x)}(T_{n(1+\sqrt\epsilon),\epsilon}(x)\cap K_\tau)}{\mu_{\xi_{i+1}(x)}(T_{n(1+\sqrt\epsilon),\epsilon}(x))}\geq e^{-\tau n}.
\end{equation*}
Let $\epsilon\in (0,\tau)$ and $n_\tau'\geq n_\tau$ s.t $\mu(A')\geq \mu(K_\tau)e^{-\frac{\tau}{2}}\geq \mu(A\cap \Lambda^{(\underline\chi,\tau)})e^{-\tau}$, where
\begin{equation}\label{keyOfKeyPrime}
A':=\{x\in K_\tau: \epsilon_x\geq \epsilon, n_x(\epsilon)\leq n_\tau\}.
\end{equation}

Let $x\in A'$. Cover $K_\tau\cap T_{n(1+\sqrt\epsilon),\epsilon}(x)$ with Besicovitch cover of transverse balls of the form $B^T(\cdot, e^{-n\chi_i\beta_{i+1}+n\tau})$, $\mathcal{B}^T$. 

Notice that for every $B^T\in\mathcal{B}^T$, if $y\in K_\tau\cap T_{n (1+\sqrt\epsilon),\epsilon}(x)\cap B^T$, then $ T_{n (1+\sqrt\epsilon),\epsilon}(x)\cap B^T \subseteq B_{\xi_{i+1}(x)}(y,e^{-\chi_i\beta_{i+1}n})$, and so $\mu_{\xi_{i+1}(x)}(T_{n (1+\sqrt\epsilon),\epsilon}(x)\cap B^T)\leq e^{-nd_i\chi_i\beta+\tau n} $. 

By Theorem \ref{Grail5}, we may also assume that $\epsilon>0$ is sufficiently small and $n$ is large enough so $\mu_{\xi_{i+1}(x)}(T_{n (1+\sqrt\epsilon),\epsilon}(x))\geq e^{-n\rho^T_{i+1}-n\tau}$. Then, since by Theorem \ref{Grail5} $\rho^T_{i+1}=n\beta_{i+1}\chi_id_i+\Delta(d_{i+1}-d_i)$, we have
$$\#\mathcal{B}^T\geq \frac{1}{C_d}\cdot\frac{e^{-n\rho^T_{i+1}-n\tau}\cdot e^{-n\tau}}{e^{-nd_{i+1}\chi_i\beta_{i+1}+\tau n}}\geq e^{n(d_{i+1}-d_i)(\chi_i\beta-\Delta)-4\tau n}.$$

Finally, notice that for all $y\in K_\tau\cap T_{n (1+\sqrt\epsilon),\epsilon}(x)\cap B^T$, for all $n$ large enough w.r.t $\Lambda^{(\underline\chi,\tau)}$, $$T_{\beta_{i+1}\chi_i,\frac{\beta_{i+1}}{1-\frac{\theta}{2}},n,\epsilon}(y)\subseteq T_{n,\epsilon}(x)\cap B^T(x_B, e^{-n\chi_i\beta_{i+1}+3n\tau}),$$ where $B^T=B^T(x_B, e^{-n\chi_i\beta_{i+1}+n\tau})$. In order to bound from below the number of disjoint elements $B^{T}(\cdot,e^{-n\chi_i\beta+3\tau n})$ which we can find in $\mathcal{B}^T$, we wish to bound their multiplicity in the cover $\mathcal{B}^T$. We refer to the computation in the tubular covers lemma, Lemma \ref{tCover}, to recall that the multiplicity for ``inflated" transverse balls is bounded by the volume estimates on the quotient space $\xi_{i+1}/\xi_i$, and so, $\forall B^T\in\mathcal{B}^T$,
$$\#\{(B^T)'\in\mathcal{B}^T:(e^{2\tau n} (B^T)')\cap (e^{2\tau n} B^T)\neq \varnothing \}\leq \widetilde{C}_d e^{3d\tau n}.$$
Thus in total, for all $n$ large enough, the statement follows. 
\end{proof}

\noindent\textbf{Remark:} Allow us to explain the idea behind the following theorem. Imagine that the distribution $E_{i+1}$ were Lipschitz continuous, and so $\theta=1$ (recall Definition \ref{tubes}). In this case, a tube could be of unbounded eccentricity, for example $B^T(x,e^{-\epsilon n})\cap \{y:|\pi_i(y)-\pi_i(x)|\leq e^{-n\chi_{i+1} }\}$, with a measure corresponding to its section by $\xi_i$-conditionals- $e^{-nd_i\chi_{i+1}+O(\epsilon n)}$. We could then divide the tube into $\sim e^{n(h_{i+1}-h_i)}$-many balls of radius $e^{-n\chi_{i+1}}$ which contain a good Pesin point, in the spirit of Proposition \ref{keyProp}. In that case, we can consider a disk which passes through the Bowen ball of the good point in each one of those balls, as they all lie in one tube we can guarantee that the disk does not ``wiggle" too much.

Starting with a tube which is long in the transverse metric and thicker than $e^{-n\chi_{i+1}}$ in the $\xi_i$-direction means that the disk which passes through the points may ``wiggle" much, as it may have to go up and down a distance larger than $e^{-n\chi_{i+1}}$ over a distance of order $e^{-n\chi_{i+1}}$.
Unfortunately, the distribution $E_{i+1}$ is merely H\"{o}lder continuous, and $\theta$ is bounded away from $1$.

The idea of the theorem below, is to use the fact that tubes have some non-zero exponential eccentricity, in order to intermediately bridge between their eccentricity and the desired eccentricity; while in each step controlling the ``wiggling" of the disk. This process relies heavily on the geometric measure theoretic properties of tubes, as a differentiation basis and their geometric description.

\begin{theorem}[Entropy gap]\label{entGapThm}
	Let $1<r\in (\mathbb{R}\setminus\mathbb{N})\cup \{\infty\}$, and assume $f\in \mathrm{Diff}^r(M)$. Let $\mu$ be an $f$-invariant ergodic measure with $u$ distinct positive Lyapunov exponents, and let $k_{i+1}$, $0\leq i\leq u-1$, be the dimension of the Oseledec subspace of the $i+1$-th largest exponent. Then,
	$$h_{i+1}-h_i\leq \mathcal{V}_{k_{i+1}}^r(f).$$
\end{theorem}
\begin{proof}\text{ }

\underline{\textbf{The case of $i=0$:}}
In that case $h_0=0$, and we consider $\xi_1(x)$ of a $\mu$-typical point $x$. For every $\delta\in (0,\frac{1}{2})$, there exist $\epsilon\in (0,\delta^2)$, a Pesin block $\Lambda^{(\underline\chi,\tau)}$ with $\tau\in (0,\epsilon^2)$, an integer $n_0\in\mathbb{N}$, and a set $K_\delta$ s.t $\mu(K_\delta)\geq 1-\delta$ where
\begin{align*}
	K_\delta:=\{x\in\Lambda^{(\underline\chi,\tau)}:\forall n\geq n_0, \mu_{\xi_1(x)}(B_{\xi_1(x)}(x,e^{-(\chi_1+\epsilon)n}))\leq e^{-nh_1+\delta n}\}.
\end{align*}
We can assume w.l.o.g that $x$ is a $\mu_{\xi_1(x)}$-density point of $K_\delta$, and cover $B_{\xi_1(x)}(x,e^{-\epsilon n})\cap K_\delta$ with a Besicovitch cover by balls $B_{\xi_1(x)}(\cdot,e^{-(\chi_1+\epsilon)n})$. Note, by \cite[Theorem~6.1]{RuelleFoliations}, $W^1(x)$ is a standard $k_1$-disk. Assume further that $\mu_{\xi_1(x)}(B_{\xi_1(x)}(x,e^{-\epsilon n}))\geq e^{-nd\epsilon}$ for all $n\geq n_1\geq n_0$. Notice that for all $y\in K_\delta\cap \xi_1(x)$, $\Vol_{k_1}(f^n[B_{\xi_1(x)}(x,e^{-(\chi_1+\epsilon)n})])\geq e^{-2\delta n\cdot k_1} $ then
\begin{align*}
	\frac{1}{C_d}\frac{(1-\delta)\cdot e^{-n\epsilon d}}{e^{-nh_1+\delta n}}\cdot e^{-2\delta n k_1}\leq \Vol_{k+1}(f^n[B_{\xi_1(x)}(x,e^{-\epsilon n})]).
\end{align*}
	Therefore, since $\delta>0$ can be arbitrarily small and $0<\epsilon<\delta$, we get $h_1\leq \mathcal{V}_{k_1}^r(f)$.
	
\medskip
\underline{\textbf{The case of $i\geq 1$:}}

This case adds a significant difficulty, which is the lack of conformality for the action of $f$ in the $\xi_i$-direction- i.e different expanding factors. To address this, we use the geometric properties of tubes. The heart of the proof is to find a tube inside which we can control the number of smaller tubes around Pesin regular points, and continuing to go into smaller and smaller scale, in a way which allows us to find a standard $k_{i+1}$-disk with a bounded norm where we control its volume growth quantitatively. The proof follows a recursive argument.

\textbf{Step 1:} We define the following starting values:
\begin{enumerate}
\item Let $\theta>0$ smaller than the H\"older exponent of the $E_{i+1}(\cdot)$ on Pesin blocks of $\mu$ (see \cite{BrinHolderContFoliations}), and small enough so $(1-\theta)\chi_i>\chi_{i+1}$.
\item Let $\kappa>0$ small.
\item $\Delta_0:=(1-\frac{\theta}{2})\kappa$.
\item $\beta^{(0)}:=\frac{\Delta_0}{(1-\frac{\theta}{2})\chi_i}\in(\frac{\Delta_0}{\chi_i}, \frac{\Delta_0}{\chi_i(1-\theta)})$.
\end{enumerate}

Let $0<\tau\ll(\frac{\kappa^2}{7d})^3$ small, let $\Lambda^{(\underline\chi,\tau)}$ be a Pesin block with a positive measure.

\textbf{Step 2:} We define the recursive values: Let $T^j_{n,\epsilon}(x)$ be a $(\Delta_j,\beta^{(j)},n,\epsilon)$-tube, where 
\begin{enumerate}
	\item $\Delta_{j+1}:=\beta^{(j)}\chi_i$,
	\item $\beta^{(j)}:=\frac{\Delta_j}{(1-\frac{\theta}{2})\chi_i}$.
\end{enumerate}
Then it follows that for  all $ j\in\mathbb{N}$,
\begin{equation}\label{forFutRef}
	\Delta_{j}:=\frac{\kappa}{(1-\frac{\theta}{2})^{j-1}}\text{, }\beta^{(j)}:=\frac{\kappa}{(1-\frac{\theta}{2})^j\chi_i}.
\end{equation}
Set
\begin{equation}\label{theNvaluation}
 	N=N(\underline\chi,\theta,\kappa):=\min\{j\in\mathbb{N}: \Delta_j\geq \chi_{i+1}\}
 	.
\end{equation}

\textbf{Step 3:} Let $x\in \Lambda^{(\underline\chi,\tau)}=:A_0 $, $n_0\in\mathbb{N}$, and $\epsilon_0>0$ given by Proposition \ref{keyProp}. Define $$L_x^{(0)}:=\exp_{x}^{W^{i+1}(x)}[E_{i+1}(x)]\cap B^T(x,e^{-\Delta_0 n-\epsilon_0 n}).$$ By the construction of the tube $T_{n,\epsilon_0}^{0}(x)$, $L_x^{(0)}\subseteq T_{n,\epsilon_0}^{0}(x)$. By \cite[Theorem~6.1]{RuelleFoliations}, $W^{i+1}(x)$ is a $C^r$ manifold, and so is the exponential map of it, hence $L_x^{(0)}$ is a $C^r$ disk. A standard result regarding the regularity of local unstable leaves guarantees that 
\\$\sup_{y\in\xi_{i+1}(x)}\|\exp_y^{W^{i+1}(x)}\|_{C^{1+\{r\}}}<\infty$. 

\textbf{Step 4:} Let $1\leq j\leq N$, and $A_{j}$, and define $A_{j+1}:=(A_j)'$ (recall \eqref{keyOfKeyPrime}), $\epsilon_{j+1}\in (0,\epsilon_j)$, and $n_{j+1}\geq n_j$ given by Proposition \ref{keyProp}
. 



Assume that $L^{(j)}_x$ is a standard $k_{i+1}$-disk in $T^{j}_{n,\epsilon_j}(x)$. We continue to modify $L_x^{(j)}$ in patches, in a way which keeps its $C ^{1+\{r\}} $-norm small. Consider all disjoint transverse balls given by Proposition \ref{keyProp}, $\mathcal{B}^T$. By construction, each one contains a strictly smaller transverse ball with the same center, and radius $e^{-n\chi_i\beta^{(j)}+\tau n}$, s.t the small transverse ball contains $x\in A_j$. For each such transverse ball $B^T$, we get for all $n
\geq n_{j+1}$, $$A_j\ni y_B\in (e^{-2\tau n}B^T)\cap T_{n(1+\sqrt\epsilon_{j+1}),\epsilon_{j+1}}^{j}(x) \Subset B^T\cap T_{n,\epsilon_{j+1}}^{j}(x).$$

The patching is done over the segment $L_x^{(j)}\cap B^T$, where we modify the disk by a ``bump" so it contains $$\exp_{y_B}^{W^{i+1}(x)}[B_{E_{i+1}(y_B)}(0,e^{-\chi_i\beta^{(j)}n-\epsilon_{j+1} n})].$$ Note, $\chi_i\beta^{(j)} =\Delta_{j+1}$, and $\sphericalangle(E_{i+1}(y_B),E_{i+1}(x))\leq e^{-(\Delta_j+\sqrt\epsilon_{j+1}) n\theta}$. The dimensions of the sets $ (e^{-2\tau n}B^T)\cap T_{n(1+\sqrt\epsilon_{j+1}),\epsilon_{j+1}}^j $ and $ B^T\cap T_{n,\epsilon_{j+1}}^{j}(x) $ guarantee that this can be done without causing the $C^{1+\{r\}}$-norm of the disk inside the box to exceed $4  \sup_{y\in \xi_{i+1}(x)}\|\exp_y^{W^{i+1}(x)}\|_{C^{1+\{r\}}} $.\footnote{In each box $B^T\cap T_{n,\epsilon_{j+1}}^{j}(x)$, the modified $L_{y_B}^{(j)}$ can be thought of as a graph of a function over $L_x^{(j)}$, with derivative bounded by $2$; If $\eta:B_{\mathbb{R}^k}(0,1)\to M$ is the disk, we can write $\eta=\exp_{y_B}^{W^{i+1}(x)}\circ \widetilde{\eta}$,
then $\|\eta\|_{C^{1+\{r\}}}\leq \|d_\cdot\widetilde{\eta}\| _{\mathrm{Lip}}^{1+\{r\}} \cdot \|\exp_{y_B}^{W^{i+1}(x)}\|_{C^{1+\{r\}}}\leq 4  \|\exp_{y_B}^{W^{i+1}(x)}\|_{C^{1+\{r\}}}$.\label{footer}} Call the modified standard $k_{i+1}$-disk $L_{y_B}^{(j+1)}$.

\textbf{Step 5:} We continue this way to modify all $L_x^{(j)}$ over all segments $\{L_{y_B}^{(j+1)}\}_{B\in \mathcal{B}^T}$, and call the total modified disk $L_x^{(j+1)}$. By the main statement of Proposition \ref{keyProp}, at each step $j$, we divide each tube of order $j$, into at least $e^{n(h_{i+1}-h_i)(\frac{\chi_i\beta^{(j)}}{\chi_{i+1}}-\frac{\Delta_j}{\chi_{i+1}})-7d\tau n}$-many tubes of order $j+1$. Therefore, together with \eqref{forFutRef}, the total number of tubes of order $N$ for all $n\geq n_N$, for all $x\in A_N$, is at least
\begin{align}\label{totalNumber}
	\exp&\left(n(h_{i+1}-h_i)\sum_{j=0}^{N-1} (\frac{\chi_i}{\chi_{i+1}}\frac{\kappa}{(1-\frac{\theta}{2})^{j}\chi_i}-\frac{1}{\chi_{i+1}}\frac{\kappa}{(1-\frac{\theta}{2})^{j-1}})-7Nd\tau n\right)\nonumber\\
	=&	\exp\left(n(h_{i+1}-h_i) \frac{1}{\chi_{i+1}}(\frac{\kappa}{(1-\frac{\theta}{2})^{N-1}}-\kappa(1-\frac{\theta}{2}))-7Nd\tau n\right)\nonumber\\
	=& \exp\left(n(h_{i+1}-h_i)\frac{1}{\chi_{i+1}}(\Delta_N-\kappa(1-\frac{\theta}{2})) -7Nd\tau n\right)\nonumber\\
	\text{(}\because\text{\eqref{theNvaluation})}\geq& \exp\left(n(h_{i+1}-h_i) (1-\kappa\frac{1-\frac{\theta}{2}}{\chi_{i+1}}) -7Nd\tau n\right)\nonumber\\
	\geq &\exp\left(n(h_{i+1}-h_i)  -\sqrt\kappa n\right).
\end{align}
The last inequality holds for all $\tau>0$ small enough w.r.t $\kappa$, $\theta$, and $\chi_{i+1}$, which determine the value of $N$; and for all $\kappa>0$ small enough w.r.t $(h_j)_{j\leq u}$ and $\underline\chi$.

\textbf{Step 6:}
Also notice that for every tube $T^{N}_{n,\epsilon_N}$ (since $\Delta_{N}< \chi_{i+1}(1-\frac{\theta}{2})$), $$\Vol_{k_{i+1}}\Big(f^{n(1-\frac{\theta}{2})}\Big[L_x^{(N)}\cap T_{n,\epsilon_N}^{(N)}(y_B
)\Big]\Big)\geq e^{-4\tau n(1-\frac{\theta}{2})k_{i+1}}\geq e^{-4\tau nk_{i+1}},$$
thus, by \eqref{totalNumber},
\begin{equation}\label{VolOfBall}
	\Vol_{k_{i+1}}\left(f^{n(1-\frac{\theta}{2})}[L_x^{(N)}]\right)\geq e^{n(h_{i+1}-h_i) -\sqrt\kappa n}\cdot  e^{-4\tau nk_{i+1}} \geq e^{n(h_{i+1}-h_i)  -2\sqrt\kappa n}.
\end{equation}

\textbf{Step 7:} 
%
Finally, $ L_{x}^{(0)}\subseteq \exp_x^{W^{i+1}(x)}[B_{E_{i+1}(x)}(0,e^{-\Delta_0 n})]$, and let $\widetilde{\eta}_n: B_{E_{i+1}(x)}(0,e^{-\Delta_0 n}) \to L_x^{(N)}$ be the representing function of the modified disk $L_x^{(N)}$ over $B_{E_{i+1}(x)}(0,e^{-\Delta_0 n}) $. Let $b: B_{E_{i+1}(x)}(0,1) \to B_{E_{i+1}(x)}(0,e^{-\Delta_0 n}) $, $t\mapsto t e^{-\Delta_0n}$, and set $\eta_n:=\widetilde{\eta}_n\circ b: B_{E_{i+1}(x)}(0,1) \to L_x^{(N)}$, then $\|\eta\|_{C^{1+\{r\}}}\leq e^{-\Delta_0n}\cdot 4  \sup_{y\in \xi_{i+1}(x)}\|\exp_y^{W^{i+1}(x)}\|_{C^{1+\{r\}}} \leq 1$, for all $n$ large enough w.r.t $\kappa$.

Therefore, 
$$h_{i+1}-h_i  -2\sqrt\kappa\leq \frac{1}{1-\frac{\theta}{2}}\mathcal{V}^r_{k_{i+1}}(f).$$ 
Since $\kappa>0$ can be arbitrarily small, $$h_{i+1}-h_i\leq \frac{1}{1-\frac{\theta}{2}}\mathcal{V}^r_{k_{i+1}}(f).$$
Since $\theta>0$ can be arbitrarily small, we are done. 
\end{proof}

\noindent\textbf{Remark:} If $f\in \mathrm{Diff}^{2+\gamma}(M)$, then the unstable leaves $W^{i+1}$ are embedded $C^{2+\gamma}$ leaves, and so the exponential map $\exp_x^{W^{i+1}(x)}$ is $C^{2+\gamma}$. By the same computation in \ref{footer}, and Step $7$, the disk $L_x^{(N)}$ will in fact have a $C^2$-norm bounded by $1$.

\begin{lemma}\label{condEntVol}
	Assume that $f\in\mathrm{Diff}^{r+\gamma}(M)$, $r\in\mathbb{N}$ and $\gamma>0$, and let $\mu$ be an ergodic $f$-invariant Borel probability with $u$ distinct positive Lyapunov exponents. Then for all $i\leq u$,
	$$h_i\leq \mathcal{V}_{\ell_i}^{r+\gamma}(f).$$
\end{lemma}
\begin{proof}
Let $i\leq u$, and let $\tau\in (0,\delta^2)$, $\epsilon\in (0,\delta^2)$, and $n_\delta$ s.t $e^{-n_\delta \epsilon^2}\ll \frac {1}{\ell_\delta}$ and  $\mu(K_\delta)>0$ where 
	\begin{align}\label{defOgKdelta}
		K_\delta:=\{x\in \Lambda^{(\underline\chi,\tau)}_\frac{1}{\ell_\delta}:\forall n\geq n_0, \mu_{\xi_i(x)}(B_{\xi_i(x)}(x,n,e^{-\epsilon n}))=e^{-nh_i\pm \delta n}  \},
	\end{align}
and $\Lambda^{(\underline\chi,\tau)}_\frac{1}{\ell_\delta}$ is a Pesin block.	
Let $x\in K_\delta$ which is a $\mu_{\xi_i(x)}$-Lebesgue density point of $K_\delta$. Assume further that for all $n\geq n_\delta'\geq n_\delta$, $\mu_{\xi_i(x)}(B_{\xi_i(x)}(x,e^{-\epsilon n}))\geq e^{n\epsilon (d+1)}$, and so for all $n$ large enough 
\begin{align}\label{muOfBall}
\mu_{\xi_i(x)}(K_\delta\cap B_{\xi_i(x)}(x,e^{-\epsilon n}))\geq e^{-\delta}e^{-n\epsilon (d+1)}.
\end{align}

Given $n\geq n_\delta'$, cover $K_\delta\cap B_{\xi_i(x)}(x,e^{-\epsilon n})$ by a cover $\mathcal{C}_n$ of sets of the form\\ $f^{-n}[B_{f^n[W_i(x)]}(f^n(\cdot),e^{-2\epsilon n})]$, with multiplicity bounded by a constant $C_k$ which depends only on the dimension of $M$. This is possible by the Besicovitch covering lemma.  

Note that $\forall B\in \mathcal{C}_n$, $B\subseteq B(y_B,n,e^{-\epsilon n})\cap \xi_i(x)$, where $y_B\in K_\delta$. Therefore for all $n$ large enough, together with \eqref{muOfBall} and \eqref{defOgKdelta} we get,
\begin{align*}
	\#\mathcal{C}_n\geq \frac{e^{-\delta}e^{-n\epsilon (d+1)}}{e^{-nh_i+n\delta}}\geq e^{nh_i-2\delta n}.
\end{align*}

It follows that for all $n$ large enough,
\begin{align*}
	\mathrm{Vol}_{\ell_i}(f^n[B_{\xi_i(x)}(x,e^{-\epsilon n})])\geq \frac{1}{C_k}\cdot e^{nh_i-2\delta n}\cdot \frac{1}{2} e^{-2\epsilon n k}\geq e^{nh_i-3\delta n}.
\end{align*}	
By \cite{RuelleFoliations}, $W_i(x)$ is $C^{r+\gamma}$-manifold, therefore in total,
	$$h_i-3\delta\leq  \mathcal{V}_{\ell_i}^{r+\gamma}(f).$$
Since $\delta>0$ was arbitrary, the statement follows.
\end{proof}

\begin{cor}\label{fin}
Let $f\in \mathrm{Diff}^{\infty}(M)$ and let $\mu$ be an ergodic measure of maximal entropy with $u$ distinct positive Lyapunov exponents, $\chi_1>\ldots >\chi_u>0$. Let $\ell_i$ be the dimension of the Oseledec subspace corresponding to $\chi_1,\ldots,\chi_i$, and assume that for all $0\leq i\leq u-1$, for $\gamma>0$, $$\mathcal{V}^{1+\gamma}_{\ell_{i+1}-\ell_i}(f)\leq \mathcal{V}^{1+\gamma}_{\ell_{i+1}}(f) - \mathcal{V}^{1+\gamma}_{\ell_{i}}(f).$$
Then, for all $i\leq u$, $$h_i= \mathcal{V}^{1+\gamma}_{\ell_{i}}(f).$$
\end{cor}
\begin{proof}
By \cite{Yomdin}, $h_\mathrm{top}(f)\geq \mathcal{V}^{1+\gamma}_u(f)$. By \cite{LedrappierYoungI}, $h_u=h_\mu(f)$. Therefore, since $\mu$ is a measure of maximal entropy,
\begin{enumerate}
\item $h_u\geq \mathcal{V}^{1+\gamma}_u(f)$.
\end{enumerate}
By Theorem \ref{entGapThm} and by our assumption, for all $i\leq u-1$,
\begin{enumerate}
\item[(2)] $h_{i+1}-h_i\leq \mathcal{V}^{1+\gamma}_{\ell_{i+1}-\ell_i}(f)\leq \mathcal{V}^{1+\gamma}_{\ell_{i+1}}(f) - \mathcal{V}^{1+\gamma}_{\ell_{i}}(f) $.
\end{enumerate}
Therefore, combining $(1)$ and $(2)$, we get that
$$h_{u-1}=h_{u}-(h_{u}-h_{u-1})\geq \mathcal{V}^{1+\gamma}_u(f)-(\mathcal{V}^{1+\gamma}_{\ell_{u}}(f) - \mathcal{V}^{1+\gamma}_{\ell_{u-1}}(f))= \mathcal{V}^{1+\gamma}_{\ell_{u-1}}(f).$$
Continuing by induction, we get for all $i\leq u$,
\begin{enumerate}
\item[(3)] $h_i\geq \mathcal{V}^{1+\gamma}_{\ell_{i}}(f) $.
\end{enumerate}
However, by Lemma \ref{condEntVol}, for all $i\leq u$,
\begin{enumerate}
\item[(4)] $h_i\leq \mathcal{V}^{1+\gamma}_{\ell_{i}}(f) $.
\end{enumerate}	
Thus, putting together (3) and (4), we conclude $$\forall i\leq u\text{, } h_i=\mathcal{V}^{1+\gamma}_{\ell_{i}}(f).$$
\end{proof}

\noindent\textbf{Remarks:}
\begin{enumerate}
\item In a follow up paper, we show that for holomorphic endomorphisms of $\mathbb{C}\mathbb{P}^k$, $k\geq 1$, the assumptions of Corollary \ref{fin} hold, and as a consequence we provide a formula for the Hausdorff dimension of the measure of maximal entropy. This gives an answer to a question of Forn\ae ss and Sibony in their list of fundamental open problems in higher dimensional complex analysis and complex dynamics (\cite[Question~2.17]{FS01}), and proves the Binder-DeMarco conjecture (\cite[Conjecture~1.3]{BinderDeMarcoConj}).
\item It follows that under the assumptions of Corollary \ref{fin}, the measure of maximal entropy not only maximizes the total entropy over all invariant measures, but also maximizes each separate conditional entropy over all invariant measures with the same index (recall Lemma \ref{condEntVol}), maximizes each entropy gap over all invariant measures with the same index (recall Theorem \ref{entGapThm}), and moreover it attains the maximal possible bounds given by the volume growth on the manifold.  
\end{enumerate}

\section*{Acknowledgements}
I would like to thank the Eberly College of Science in the Pennsylvania State University for excellent working conditions. I wish to deeply thank Prof. Federico Rodriguez-Hertz for many enlightening and useful discussions, from which I have learned a lot. 

\bibliographystyle{alpha}
\bibliography{Elphi}

\Addresses

\end{document}